 \documentclass[10pt,twocolumn,twoside]{IEEEtran}

\usepackage{amsmath,amsfonts,graphicx,amssymb}
\usepackage{bm, bbm}
\usepackage{float}
\usepackage{dsfont}

\usepackage{latexsym}
\usepackage{graphicx}
\usepackage{amsthm}
\usepackage{color}
\usepackage{bbm}

\newtheorem{theorem}{Theorem}[section]
\newtheorem{lemma}{Lemma}[section]

\newcommand{\ignore}[1]{}

\title{
Concentration bounds for two time scale stochastic approximation
}

\author{Vivek S. Borkar and Sarath Pattathil 
\thanks{V.S. Borkar is with the Department of Electrical Engineering, IIT Bombay,
Powai, Mumbai 400076, India {\tt borkar.vs@gmail.com}. Work supported in part by a J.\ C.\ Bose Fellowship and a CEFIPRA Grant for `Machine Learning for Network Analytics'.}%
\thanks{S. Pattathil is with the Department of Electrical Engineering, IIT Bombay,
Powai, Mumbai 400076, India {\tt sarathpattathil@iitb.ac.in}}%
}

\begin{document}

\maketitle
\thispagestyle{empty}
\pagestyle{empty}

\begin{abstract}
Viewing a two time scale stochastic approximation scheme as a noisy discretization of a singularly perturbed differential equation, we obtain a concentration bound for its iterates that captures its behavior with quantifiable high probability. This uses Alekseev's nonlinear variation of constants formula and a martingale concentration inequality, and extends the corresponding results for single time scale stochastic approximation.
\end{abstract}

\section{Introduction}

Consider the two time scale stochastic approximation:
\begin{align}
\label{eq:stoc_approx_x}
x_{k+1} &= x_{k} + a_k \big( h(x_k,y_k) + M_{k+1}^{(1)} \big), \\
\label{eq:stoc_approx_y}
y_{k+1} &= y_{k} + b_k \big( g(x_k,y_k) + M_{k+1}^{(2)} \big), 
\end{align}
where $\{a_n\}, \{b_n\} \subset (0, 1)$ are stepsizes satisfying\footnote{As in \cite{gugan}, we can relax the second of these conditions to $a_n, b_n \to 0$. We do not discuss this here.}
\begin{align}
\label{eq:two_timescales}
\sum_na_n = \sum_nb_n = \infty,  \  \sum_n(a_n^2 + b_n^2) < \infty,  \  \frac{b_n}{a_n} \xrightarrow[]{n \uparrow \infty} 0. 
\end{align}
For simplicity, we assume $0 < b_n \leq a_n < 1$ $\forall \: n$.
These are expected to track the singularly perturbed ordinary differential equation (ODE)
\begin{align}
\dot{x}(t) = h(x(t),y(t)),  \label{odex} \\
\dot{y}(t) = \epsilon g(x(t),y(t)),  \label{odey}
\end{align}
where $0 < \epsilon \downarrow 0$. Assume that for  fixed $y$, the ODE 
\begin{align}
\dot{\tilde{x}}(t) = h(\tilde{x}(t), y) \label{odex1}
\end{align}
has a globally asymptotically stable equilibrium $\lambda(y)$ and 
\begin{align}
\dot{\tilde{y}}(t) = \epsilon g(\lambda(\tilde{y}(t)), \tilde{y}(t)) \label{odey1}
\end{align}
has a globally asymptotically stable equilibrium $y^*$.  The intuition behind (\ref{eq:stoc_approx_x})-(\ref{eq:stoc_approx_y}) is as follows. Interpretating them as noisy Euler scheme for (\ref{odex})-(\ref{odey}) (see, e.g., \cite{book}), $a_k, b_k$ can be viewed as discrete time steps. Then  the last condition in (\ref{eq:two_timescales}) induces a time scale separation whereby $\{x_k\} $ evolves on a faster time scale compared to $\{y_k\}$, thereby mimicking  (\ref{odex})-(\ref{odey}). The fast time scale sees the slow one as quasi-static, i.e., $y(t) \approx$ a constant $y$, whence $x(t)$ tracks $\lambda(y(t))$. In turn, $y(\cdot)$ approximately follows (\ref{odey1}). Hence we expect a.s.\  convergence  of $(x_k, y_k)$ to $(\lambda(y^*), y^*)$ as $k\uparrow\infty$ (\cite{book}, Chapter 6). 

The above behavior emulates nested iterations where one would perform the $x_k$ iteration till  near-convergence as a subroutine between two updates of $y_k$. The incremental online nature of the two time scale scheme makes it ideal for applications such as reinforcement learning \cite{SPP}, \cite{LeslieCollins}, \cite{KondaBorkar}, \cite{KondaTsitsiklis}. While the convergence analysis sketched above is by now classical \cite{two_time_scale},  the convergence rate and error analysis is lacking except in the linear case \cite{gugan_2}, \cite{konda}. 

The aim of the present work is to provide a concentration result for the two time scale scheme in the spirit of \cite{gugan}, which does so for the single time scale case. This may be viewed as a step towards the aforementioned objective. We refer to \textit{ibid.} for details that are common, focusing only on the points of departure. We make the following assumptions throughout:
\begin{itemize}
\item $h(\cdot) : \mathbb{R}^d\times \mathbb{R}^s \mapsto \mathbb{R}^d$, $g(\cdot) : \mathbb{R}^d\times\mathbb{R}^s \mapsto \mathbb{R}^s$, $\lambda(\cdot) : \mathbb{R}^s \mapsto \mathbb{R}^d$ are Lipschitz with Lipschitz constants $L_h, L_g, L_\lambda$ respectively, $\|g(\cdot)\| \leq B_g < \infty$.

\item $M_{n+1}^{(i)}, i = 1,2,$ are martingale difference sequences with respect to the increasing $\sigma-$fields
\begin{align*}
\mathcal{F}_n := \sigma(x_m, y_m, M_m^{(i)}; i = 1,2; m \leq n), \ n \geq 0. 
\end{align*}
That is,  $\mathbb{E}[M_{n+1}^{(i)} | \mathcal{F}_n] = 0$ a.s.\ $\forall \  i = 1,2; \ n \geq 0.$
Furthermore, $\exists \ c_1, c_2: \mathbb{R}^d \rightarrow (0, \infty), u_L > 0$, such that $\forall i =  1,2; n \geq 0; u > u_L;$, 
\begin{align*}
\mathbb{P}\{ ||M_{n+1}^{(i)} || > u | \mathcal{F}_n\} \leq c_1(x_n) e^{-c_2(x_n) u}.
\end{align*}
\end{itemize}

We also assume \textbf{A4} from \cite{gugan}:  There exist a Lyapunov function $V \in C^1(\mathbb{R}^s)$ with $\lim_{\|y\|\uparrow\infty}V(y) = \infty$, $\langle\nabla V(y), g(\lambda(y), y)\rangle < 0$ for $y \neq y^*$, and $r > r_0 > 0, \epsilon_0 > 0$ such that for $\epsilon < \epsilon_0$,
\begin{eqnarray*}
\lefteqn{\{ y \in \mathbb{R}^s: ||y - y^*|| \leq \epsilon \} } \\
&& \subseteq V^{r_0}  := \{ y \in \text{dom} (V) : V(y) \leq r_0\} \\
&&\subset \mathcal{N}_{\epsilon_0}(V^{r_0}) \subseteq V^r \subset \text{dom} (V), 
\end{eqnarray*}
where $V^{r}$ is defined similarly to $V^{r_0}$ and
\begin{align*}
\mathcal{N}_{\epsilon_0}(V^{r_0}) := \{ y \in \mathbb{R}^d: \exists y' \in V^{r_0}  \ \text{ s.t } ||y' - y|| \leq \epsilon_0 \}.
\end{align*}
We make an analogous assumption for each fixed $y$ and for the equilibrium $\lambda(y)$ of (\ref{odex1}), with $x, \lambda(y)$ replacing $y, y^*$ in the above. We use the common notation $V(\cdot)$ for the Liapunov function of both cases, suppressing the $y$-dependence in the latter. Note that this `assumption'  is in fact guaranteed by the converse Lyapunov theorem \cite{Krasovskii}.

\section{Alekseev's Formula}

Alekseev's formula \cite{alekseev} is a nonlinear variation of constants formula for nonlinear ODE.  We give a slightly more general form from \cite{aseem} that allows for differing initial conditions.
\begin{theorem}
Consider a differential equation 
\begin{align}
\dot{u}(t) = f(t, u(t)), \: t \geq 0, \nonumber
\end{align}
and its perturbation
\begin{align}
\dot{p}(t) = f(t, p(t)) + g(t, p(t)), \: t\geq 0, \nonumber
\end{align}
where $f, g: \mathbb{R} \times \mathbb{R} \rightarrow \mathbb{R}^d$, $f \in C^1(\mathbb{R}^d), g \in C(\mathbb{R}^d)$.  Let $u(t, t_0, p_0)$ and $p(t, t_0, p_0)$ denote respectively the solutions to the above nonlinear systems for $t \geq t_0$ satisfying $p(t_0, t_0, p_0) = p_0, u(t_0, t_0, p_0) = u_0$. Then
\begin{align}
p(t, t_0, &p_0) = u(t, t_0, u_0) + \Phi(t, t_0, p_0)(p_0-u_0) \nonumber \\
& + \int_{t_0}^{t} \Phi(t, s, p(s, t_0, p_0))g(s, p(s, t_0, p_0))ds, \: \: t\geq t_0. \nonumber 
\end{align} 
Here $\Phi(t,s,u_0)$ for $t \geq s, u_0 \in \mathbb{R}^d$, is the fundamental matrix of the linear system 
\begin{align*}
\dot{v}(t) = \frac{\partial f}{\partial u} (t, u(t, s, u_0))v(t), \: \: t \geq s, 
\end{align*}
with $\Phi(s,s,u_0) = \mathcal{I}_d :=$ the $d-$dimensional identity matrix. \\

\end{theorem}
 
\section{Error Bounds}
In what follows, $K \in (0, \infty)$ will denote a generic constant depending on the context.  Let $z_k = \lambda(y_k)$, i.e., $h(z_k, y_k) = 0, \ k \geq 0$.
Let $\nabla \lambda :=$ the Jacobian matrix of $\lambda(\cdot)$. Using Taylor expansion, a `stochastic approximation scheme'  for $\{z_k\}$ can be written as
\begin{align}
z_{k+1} = z_{k} + \nabla \lambda(y_k) (y_{k+1} - y_k) + \zeta_{k+1}. \nonumber
\end{align}
Here $||\zeta_{k+1}|| \leq K_\zeta ||y_{k+1} - y_k ||^2$ is the error term from Taylor expansion. Substituting from (\ref{eq:stoc_approx_y}), we get:
\begin{align}
z_{k+1} &= z_{k} + \nabla \lambda(y_k) \big( b_k g(x_k,y_k) + b_k M_{k+1}^{(2)} \big) + \zeta_{k+1} \nonumber \\
&= z_{k} + a_k h(z_k,y_k) \nonumber \\
&\qquad + \nabla \lambda(y_k) \big( b_k g(x_k,y_k) + b_k M_{k+1}^{(2)} \big) + \zeta_{k+1}, \nonumber
\end{align}
because $h(z_k,y_k) = h(\lambda(y_k),y_k) = 0$.
This leads to:
\begin{align}
z_{k+1} = z_{k} + a_{k} \big( h(z_k,y_k) &+ \epsilon_{k} \nabla \lambda(y_k) M_{k+1}^{(2)} \nonumber \\
&+ \eta_k  \nabla \lambda(y_k) + \varepsilon_{k+1} \big), \nonumber
\end{align}
where
\begin{align}
\epsilon_k = \frac{b_k}{a_k}, \ \eta_k = \epsilon_k g(x_k,y_k),  \
\varepsilon_{k+1} = \frac{1}{a_k} \zeta_{k+1}. \nonumber
\end{align}
We can bound $||\varepsilon_{k+1}||$ as
\begin{eqnarray*}
\lefteqn{||\varepsilon_{k+1}|| = \frac{1}{a_k} ||\zeta_{k+1}|| 
\leq \frac{1}{a_k} K_\zeta ||y_{k+1} - y_k||^2} \\
&=& \epsilon_k b_k K_\zeta ||g(x_k,y_k) + M_{k+1}^{(2)}||^2  \\
&\leq& \epsilon_k b_k K_\zeta \big( B_g^2 + 2 B_g ||M_{k+1}^{(2)}|| + ||M_{k+1}^{(2)}||^2 \big)  \\
&\leq& K\big( \epsilon_k b_k +  \epsilon_k b_k||M_{k+1}^{(2)}|| + \epsilon_k b_k||M_{k+1}^{(2)}||^2 \big). 
\end{eqnarray*}
Consider the coupled iterations:
\begin{align}
\label{eq:stoc_x} 
x_{k+1} &= x_{k} + a_k \big( h(x_k,y_k) + M_{k+1}^{(1)} \big), \\
\label{eq:stoc_z}
z_{k+1} &= z_{k} + a_{k} \big( h(z_k,y_k) + \epsilon_{k} \nabla \lambda(y_k) M_{k+1}^{(2)} \nonumber \\
& \qquad \qquad \qquad \qquad \qquad + \eta_k  \nabla \lambda(y_k) + \varepsilon_{k+1} \big).
\end{align}
As shown below, a suitable interpolation of  (\ref{eq:stoc_x})-(\ref{eq:stoc_z}) can be considered as a perturbation of the differential equations
\begin{align}
\dot{x}(t) = h(x(t),y(t)), \ \dot{y}(t) = 0,
\end{align}
facilitating an application of Alekseev's formula.

\subsection{Deviation bound for  $\{x_n\}$}

Let $\tilde{t}_0 = 0, \tilde{t}_{k+1} = \tilde{t}_k + a_k$ for $k \geq 0$.
Define interpolation $\overline{x}(\cdot)$ of $\{x_n\}$ by: $\overline{x}(\tilde{t}_k) = x_k \ \forall k$ and  for $t \in (\tilde{t}_{k}, \tilde{t}_{k+1})$ 
\begin{align}
\overline{x}(t) = x_k + \frac{t-\tilde{t}_k}{a_k} [ x_{k+1} - x_k ]. \nonumber
\end{align}
Define the event $G_n$ by\footnote{We later use the same notation $G_n$ for the event defined above for the variables $\{y_k\}$ as well. The usage will be clear from the context.} 
\begin{align}
G_n := \{ \overline{x}(t) \in V^r \: \forall \ t \in [\tilde{t}_{n_0}, \tilde{t}_n] \}. \nonumber 
\end{align}
We have
\begin{align}
\overline{x}(\tilde{t}_{n+1}) = \overline{x}(\tilde{t}_{n_0}) + \sum_{k = {n_0}}^{n} a_k h(x_k,y_k) + \sum_{k = {n_0}}^{n} a_k M_{k+1}^{(1)} . \nonumber
\end{align}
Rewrite this equation as
\begin{align}
\overline{x}(t) = \overline{x}(\tilde{t}_{n_0}) + \int_{\tilde{t}_{n_0}}^{t} \left(h(\overline{x}(s),{y}(s))ds 
+ \big( \xi_1(s) + \xi_2(s) \big)\right) ds \nonumber
\end{align}
where for $s \in [\tilde{t}_k,\tilde{t}_{k+1})$,
\begin{align}
\xi_1(s) = h(\overline{x}(\tilde{t}_k),{y}(\tilde{t}_k)) - h(\overline{x}(s),{y}(s)), \ \xi_2(s) = M_{k+1}^{(1)}. \nonumber
\end{align}
Using the generalized  Alekseev's formula above, we have:
\begin{align}
\label{eq:imp_1}
\overline{x}(t) = x(&t) + \Phi_x(t, s, \overline{x}(\tilde{t}_{n_0}), y(\tilde{t}_{n_0}))(\overline{x}(\tilde{t}_{n_0}) - {x}(\tilde{t}_{n_0}))     \nonumber \\
&+ \int_{t_{n_0}}^{t} \Phi_x(t, s, \overline{x}(s), y(s)) \big[ \xi_1(s) + \xi_2(s) \big] ds.
\end{align}
Here $y(t) \equiv y, x(t) \equiv  \lambda(y)$ is a constant  trajectory and 
$\Phi_x(\cdot)$ satisfies the linear system:
\begin{align}
\label{eq:Phi_diff_eq}
\dot{\Phi}_x(t, s,  x_0, y_0) = D(x(t), y(t)) \Phi_x(t, s, x_0, y_0), \: \: t \geq s, 
\end{align}
with initial condition $\Phi_x(t, s, x_0, y_0) = \mathcal{I}$, where $D$ is the Jacobian matrix of $h(\cdot,y)$. 
As shown in Lemma $5.3$, \cite{gugan},  there exist $K, \kappa_x > 0$ so that the following holds for $t \geq s$ and $x_0 \in V^r$:
\begin{align}
|| \Phi_x(t, s, x_0, y_0) || \leq K e^{-\kappa_x (t-s)}. \nonumber
\end{align} 
From Lemma $5.8$ \cite{gugan} and (\ref{eq:imp_1}), we have on $G_n$
\begin{align}
&||\overline{x}(\tilde{t}_n) - x(\tilde{t}_n)|| \leq ||\Phi(\tilde{t}_n, \tilde{t}_{n_0},\overline{x}(\tilde{t}_{n_0}),y)(\overline{x}(\tilde{t}_{n_0}) - x(\tilde{t}_{n_0}))|| \nonumber \\
& + K \bigg[ ||S_n^{(1)}|| + \sup_{{n_0} \leq k \leq n-1} a_k + \sup_{{n_0} \leq k \leq n-1} a_k ||M_{k+1}^{(1)}||^2 \bigg] \nonumber
\end{align}
where
\begin{align}
S_n^{(1)} = \sum_{k={n_0}}^{n-1} \bigg( \int_{\tilde{t}_k}^{\tilde{t}_{k+1}} \Phi_x(\tilde{t}_n, s, \overline{x}(\tilde{t}_k), y(\tilde{t}_k))ds  \bigg) M_{k+1}^{(1)}. \nonumber
\end{align}
This gives the following error bound: on $G_n$,
\begin{align}
&||\overline{x}(\tilde{t}_n) - x(\tilde{t}_{n})|| \leq K \bigg[ e^{-\kappa_x (\tilde{t}_n-\tilde{t}_{n_0})}||\overline{x}(\tilde{t}_{n_0}) - x(\tilde{t}_{n_0})|| \nonumber \\
& + ||S_n^{(1)}|| + \sup_{{n_0} \leq k \leq n-1} a_k + \sup_{{n_0} \leq k \leq n-1} a_k ||M_{k+1}^{(1)}||^2 \bigg]. \nonumber 
\end{align}

\subsection{Deviation bound for  $\{z_n\}$}

Define $\overline{z}(t)$ by: for $t \in (\tilde{t}_{k}, \tilde{t}_{k+1})$ 
\begin{align}
\overline{z}(t) = z_k + \frac{t-\tilde{t}_k}{a_k} [ z_{k+1} - z_k ] \nonumber
\end{align}
where $\overline{z}(\tilde{t}_k) = z_k \ \forall k$.  
We have:
\begin{align}
\overline{z}(t) = z(\tilde{t}_{n_0}) &+ \int_{\tilde{t}_{n_0}}^{t} h(\overline{z}(s),{y}(s))ds \nonumber \\
&+ \int_{\tilde{t}_{n_0}}^{t} \big( \xi_3(s) + \xi_4(s) + \xi_5(s) + \xi_6(s) \big) ds \nonumber
\end{align}
where for $s \in [\tilde{t}_k,\tilde{t}_{k+1})$,
\begin{align*}
\xi_3(s) &= h(\overline{z}(\tilde{t}_k),{y}(\tilde{t}_k)) - h(\overline{z}(s),{y}(s)), \nonumber \\
\xi_4(s) &= \epsilon_{k} \nabla \lambda(y_k) M_{k+1}^{(2)},  
\xi_5(s) = \eta_k  \nabla \lambda(y_k), 
\xi_6(s) = \varepsilon_{k+1}. 
\end{align*}
Using the generalized Alekseev's formula with $x(t) \equiv \lambda(y)$ and $\Phi_x$ as in (\ref{eq:Phi_diff_eq}), we have
\begin{align*}
&\overline{z}(\tilde{t}_n) = x(\tilde{t}_n) + \Phi_x(\tilde{t}_n, \tilde{t}_{n_0}, \overline{z}(\tilde{t}_{n_0}), y(\tilde{t}_{n_0})) (\overline{z}(\tilde{t}_{n_0}) - x(\tilde{t}_{n_0})) \\
&+ A_n + B_n+ C_n + D_n, 
\end{align*}
where
\begin{align}
A_n &= \sum_{k={n_0}}^{n-1} \int_{\tilde{t}_k}^{\tilde{t}_{k+1}} \Phi_x(\tilde{t}_n, s, \overline{z}(s), y(s)) \big[ h(\overline{z}(\tilde{t}_k),{y}(\tilde{t}_k)) \nonumber \\
& \qquad \qquad \qquad \qquad \qquad \qquad \qquad - h(\overline{z}(s),{y}(s)) \big]ds, \nonumber \\
B_n &= \sum_{k={n_0}}^{n-1} \int_{\tilde{t}_k}^{\tilde{t}_{k+1}} \Phi_x(\tilde{t}_n, s, \overline{z}(s), y(s)) \epsilon_{k} \nabla \lambda(y_k) M_{k+1}^{(2)}ds, \nonumber \\
C_n &= \sum_{k={n_0}}^{n-1} \int_{\tilde{t}_k}^{\tilde{t}_{k+1}} \Phi_x(\tilde{t}_n, s, \overline{z}(s), y(s)) \eta_k  \nabla \lambda(y_k) ds, \nonumber \\
D_n &= \sum_{k={n_0}}^{n-1} \int_{\tilde{t}_k}^{\tilde{t}_{k+1}} \Phi_x(\tilde{t}_n, s, \overline{z}(s), y(s)) \varepsilon_{k+1} ds. \nonumber
\end{align}
As in  the previous subsection, we have
\begin{align}
||\Phi_x(\tilde{t}_n, \tilde{t}_{n_0}, \overline{z}(\tilde{t}_{n_0})&, y(\tilde{t}_{n_0})) (\overline{z}(\tilde{t}_{n_0}) - x(\tilde{t}_{n_0}))|| \nonumber \\
&\leq e^{-\kappa_x (\tilde{t}_n-\tilde{t}_{n_0})}||\overline{z}(\tilde{t}_{n_0}) - x(\tilde{t}_{n_0})||. \nonumber
\end{align}
We bound other terms through a sequence of lemmas. \\
\begin{lemma}
\label{lemma:z_A_n_prereq}
Let $k,n$ with $n_0 \leq k \leq k+1 \leq n$ be arbitrary. Then on $G_n$,
\begin{align}
&\int_{\tilde{t}_k}^{\tilde{t}_{k+1}} e^{-\kappa_x(\tilde{t}_n-s)}||\overline{z}(s) - \overline{z}(\tilde{t}_k)|| \leq K \bigg(\epsilon_k + \epsilon_k ||M_{k+1}^{(2)}||   \nonumber \\
&+ \epsilon_k b_k + \epsilon_k b_k ||M_{k+1}^{(2)}|| + \epsilon_k b_k ||M_{k+1}^{(2)}||^2 \bigg) e^{-\kappa_x(\tilde{t}_n-\tilde{t}_{k+1})}a_k^2. \nonumber
\end{align}
\end{lemma}

\begin{proof}
We have:
\begin{align}
||&\overline{z}(s) - \overline{z}(\tilde{t}_k)|| = \frac{(s-\tilde{t}_k)}{a_k} ||\overline{z}(\tilde{t}_{k+1}) - \overline{z}(\tilde{t}_k)|| \nonumber \\
&= \frac{(s-\tilde{t}_k)}{a_k} ||\nabla \lambda(y_k) (y_{k+1} - y_k) + \zeta_{k+1}|| \nonumber \\
&\leq \frac{(s-\tilde{t}_k)}{a_k} (L_\lambda ||y_{k+1} - y_k|| + K_\zeta||y_{k+1} - y_k||^2) \nonumber \\
&\leq K \frac{(s-\tilde{t}_k)}{a_k} (||y_{k+1} - y_k|| + ||y_{k+1} - y_k||^2) \nonumber \\
&\leq K (s-\tilde{t}_k) (\epsilon_k B_g + \epsilon_k ||M_{k+1}^{(2)}|| + \epsilon_k b_k B_g^2 \nonumber \\
& \qquad \qquad \qquad \qquad +2\epsilon_k b_k B_g ||M_{k+1}^{(2)}|| + \epsilon_k b_k ||M_{k+1}^{(2)}||^2) \nonumber \\
&\leq K (s-\tilde{t}_k) (\epsilon_k + \epsilon_k ||M_{k+1}^{(2)}|| + \epsilon_k b_k + \epsilon_k b_k ||M_{k+1}^{(2)}|| \nonumber \\
&\qquad \qquad \qquad \qquad  + \epsilon_k b_k ||M_{k+1}^{(2)}||^2). 
\label{eq:z_A_L1_1}
\end{align}
The result now follows from (\ref{eq:z_A_L1_1}) and
\begin{align*}
\int_{\tilde{t}_k}^{\tilde{t}_{k+1}}(s-\tilde{t}_k) e^{-\kappa_x(\tilde{t}_n-s)}ds \leq e^{-\kappa_x(\tilde{t}_n-\tilde{t}_{k+1})}a_k^2.
\end{align*}
\end{proof}

\begin{lemma}
Let $n \geq n_0$ be arbitrary. Then on $G_n$,
\begin{align}
||A_n|| \leq& K \big[ \sup_{{n_0} \leq k \leq n-1} b_k + \sup_{{n_0} \leq k \leq n-1} b_k ||M_{k+1}^{(2)}||  \nonumber \\
&+ \sup_{{n_0} \leq k \leq n-1} b_k^2 +  \sup_{{n_0} \leq k \leq n-1} b_k^2||M_{k+1}^{(2)}|| \nonumber \\
&+ \sup_{{n_0} \leq k \leq n-1} b_k^2||M_{k+1}^{(2)}||^2 \big]. \nonumber
\end{align}
\end{lemma}

\begin{proof}
The proof mimics that of Lemma $5.6$, \cite{gugan}. Thus,
\begin{align}
||A_n|| \leq& \sum_{k={n_0}}^{n-1} \int_{\tilde{t}_k}^{\tilde{t}_{k+1}} ||\Phi_x(\tilde{t}_n, s, \overline{z}(s), y(s))|| \times \nonumber \\
& \qquad \qquad \qquad || h(\overline{z}(\tilde{t}_k),{y}(\tilde{t}_k)) - h(\overline{z}(s),{y}(s)) ||ds \nonumber \\
\leq& L_h\sum_{k={n_0}}^{n-1}\int_{\tilde{t}_k}^{\tilde{t}_{k+1}}|| \Phi_x(\tilde{t}_n, s, \overline{z}(s), y(s))|| \times \nonumber \\
&\qquad \qquad \qquad \qquad \qquad \qquad || \overline{z}(\tilde{t}_k) - \overline{z}(s) ||ds \nonumber \\
\leq& K \sum_{k={n_0}}^{n-1} \int_{\tilde{t}_k}^{\tilde{t}_{k+1}} e^{-\kappa_x (\tilde{t}_n-s)} ||\overline{z}(\tilde{t}_k) - \overline{z}(s)||ds \nonumber \\
\leq& K \sum_{k={n_0}}^{n-1} \bigg(\epsilon_k + \epsilon_k ||M_{k+1}^{(2)}|| + \epsilon_k b_k + \epsilon_k b_k ||M_{k+1}^{(2)}|| \nonumber 
\end{align}
\begin{align}
&+ \epsilon_k b_k ||M_{k+1}^{(2)}||^2 \bigg) e^{-\kappa_x(\tilde{t}_n-\tilde{t}_{k+1})}a_k^2 \nonumber \\
\leq & K \big[ \sup_{{n_0} \leq k \leq n-1} b_k + \sup_{{n_0} \leq k \leq n-1} b_k ||M_{k+1}^{(2)}||  \nonumber \\
&+ \sup_{{n_0} \leq k \leq n-1} b_k^2 +  \sup_{{n_0} \leq k \leq n-1} b_k^2||M_{k+1}^{(2)}||  \nonumber \\
&+ \sup_{{n_0} \leq k \leq n-1} b_k^2||M_{k+1}^{(2)}||^2 \big] \times  \sum_{k={n_0}}^{n-1} e^{-\kappa_x(\tilde{t}_n-\tilde{t}_{k+1})}a_k. \nonumber
\end{align}
The claim follows on observing that (since $a_k < 1$)
\begin{align}
\sum_{k={n_0}}^{n-1} e^{-\kappa_x(\tilde{t}_n-\tilde{t}_{k+1})}a_k &\leq  e^{\kappa_x}\int_{\tilde{t}_{n_0}}^{\tilde{t}_n }e^{-\kappa_x(\tilde{t}_n-s)}ds \leq \frac{e^{\kappa_x}}{\kappa_x}. \nonumber
\end{align}
\end{proof}

\begin{lemma}
Let $n \geq n_0$ be arbitrary. Then on $G_n$,
\begin{align}
||B_n|| \leq K \big[ \sup_{{n_0} \leq k \leq n-1} \epsilon_k || M_{k+1}^{(2)} || \big]. \nonumber
\end{align}
\end{lemma}

\begin{proof}
We have:
\begin{align}
||B_n|| &\leq \sum_{k={n_0}}^{n-1} \int_{\tilde{t}_k}^{\tilde{t}_{k+1}} ||\Phi(\tilde{t}_n, s, \overline{z}(s), y(s))|| \times \nonumber \\
&\qquad \qquad \qquad \qquad \qquad ||\epsilon_{k} \nabla \lambda(y_k) M_{k+1}^{(2)} ||ds \nonumber \\
&\leq L_\lambda \sum_{k={n_0}}^{n-1} \int_{\tilde{t}_k}^{\tilde{t}_{k+1}} ||\Phi(\tilde{t}_n, s, \overline{z}(s), y(s))|| \times \nonumber \\
&\qquad \qquad \qquad \qquad \qquad ||\epsilon_{k} M_{k+1}^{(2)} ||ds \nonumber \\
&\leq  K \sum_{k={n_0}}^{n-1} \epsilon_{k} || M_{k+1}^{(2)} || \int_{\tilde{t}_k}^{\tilde{t}_{k+1}} e^{-\kappa_x(\tilde{t}_n-s)} ds \nonumber \\
&\leq  K \big[ \sup_{{n_0} \leq k \leq n-1} \epsilon_k || M_{k+1}^{(2)} || \big] \sum_{k={n_0}}^{n-1} e^{-\kappa_x(\tilde{t}_n-\tilde{t}_{k+1})}a_k \nonumber \\
&\leq K \big[ \sup_{{n_0} \leq k \leq n-1} \epsilon_k || M_{k+1}^{(2)} || \big]. \nonumber
\end{align}
\end{proof}

\begin{lemma}
Let $n \geq n_0$ be arbitrary. Then on $G_n$,
\begin{align}
||C_n|| \leq K \big[\sup_{{n_0} \leq k \leq n-1} \epsilon_k \big]. \nonumber
\end{align}
\end{lemma}

\begin{proof}
We have:
\begin{align}
||C_n|| &\leq \sum_{k={n_0}}^{n-1} \int_{\tilde{t}_k}^{\tilde{t}_{k+1}} ||\Phi(\tilde{t}_n, s, \overline{z}(s), y(s))|| \times \nonumber \\
& \qquad \qquad \qquad \qquad \qquad ||\eta_k  \nabla \lambda(y_k)|| ds \nonumber \\
&\leq L_\lambda B_g \sum_{k={n_0}}^{n-1} \int_{\tilde{t}_k}^{\tilde{t}_{k+1}} \epsilon_k ||\Phi(\tilde{t}_n, s, \overline{z}(s), y(s))||ds \nonumber \\
&\leq  K \sum_{k={n_0}}^{n-1} \epsilon_k \int_{\tilde{t}_k}^{\tilde{t}_{k+1}} e^{-\kappa_x(\tilde{t}_n-s)} ds \nonumber \\
&\leq  K \big[\sup_{{n_0} \leq k \leq n-1} \epsilon_k \big]\sum_{k=0}^{n-1} e^{-\kappa_x(\tilde{t}_n-\tilde{t}_{k+1})}a_k \nonumber \\
&\leq K \big[\sup_{{n_0} \leq k \leq n-1} \epsilon_k \big]. \nonumber
\end{align}
\end{proof}
\begin{lemma}
Let $n \geq n_0$ be arbitrary. Then on $G_n$,
\begin{align}
||D_n|| \leq K \bigg[\sup_{n_0 \leq k \leq n-1} \epsilon_k b_k &+ \sup_{n_0 \leq k \leq n-1} \epsilon_k b_k || M_{k+1}^{(2)}|| \nonumber \\
&+ \sup_{n_0 \leq k \leq n-1} \epsilon_k b_k || M_{k+1}^{(2)}||^2 \bigg]. \nonumber
\end{align}
\end{lemma}

\begin{proof}
We have:
\begin{eqnarray*}
\lefteqn{||D_n|| \leq \sum_{k={n_0}}^{n-1} \int_{\tilde{t}_k}^{\tilde{t}_{k+1}} ||\Phi(\tilde{t}_n, s, \overline{z}(s), y(s))|| \times || \varepsilon_{k+1}||ds} \nonumber \\
&\leq K \bigg( \sum_{k={n_0}}^{n-1} \int_{\tilde{t}_k}^{\tilde{t}_{k+1}} \epsilon_k b_k ||\Phi(\tilde{t}_n, s, \overline{z}(s), y(s))||  ds  + \nonumber \\
& \sum_{k={n_0}}^{n-1} \int_{\tilde{t}_k}^{\tilde{t}_{k+1}} \epsilon_k b_k ||\Phi(\tilde{t}_n, s, \overline{z}(s), y(s))|||| M_{k+1}^{(2)}||ds  + \nonumber \\
& \sum_{k={n_0}}^{n-1} \int_{\tilde{t}_k}^{\tilde{t}_{k+1}} \epsilon_k b_k ||\Phi(\tilde{t}_n, s, \overline{z}(s), y(s))|||| M_{k+1}^{(2)}||^2ds \bigg).   \nonumber \\
\end{eqnarray*}
We bound each of these terms individually. Thus
\begin{align}
&K\sum_{k={n_0}}^{n-1} \int_{\tilde{t}_k}^{\tilde{t}_{k+1}} \epsilon_k b_k ||\Phi(\tilde{t}_n, s, \overline{z}(s), y(s))||  ds \nonumber \\
\leq& K \sum_{k={n_0}}^{n-1} \epsilon_k b_k \int_{\tilde{t}_k}^{\tilde{t}_{k+1}} e^{-\kappa_x (\tilde{t}_n -s)}ds \nonumber \\
\leq& K   \big[\sup_{n_0 \leq k \leq n-1} \epsilon_k b_k \big] \sum_{k={n_0}}^{n-1} e^{-\kappa_x (\tilde{t}_n -\tilde{t}_{k+1})}a_k \nonumber \\
\leq& K \big[\sup_{n_0 \leq k \leq n-1} \epsilon_k b_k \big], \nonumber 
\end{align}
\begin{align}
& K\sum_{k={n_0}}^{n-1} \int_{\tilde{t}_k}^{\tilde{t}_{k+1}} \epsilon_k b_k ||\Phi(\tilde{t}_n, s, \overline{z}(s), y(s))|||| M_{k+1}^{(2)}||ds \nonumber \\
\leq& K \sum_{k={n_0}}^{n-1} \int_{\tilde{t}_k}^{\tilde{t}_{k+1}} \epsilon_k b_k e^{-\kappa (\tilde{t}_n -s)} || M_{k+1}^{(2)}||ds \nonumber \\
\leq& K  \sum_{k={n_0}}^{n-1} \epsilon_k b_k e^{-\kappa (\tilde{t}_n - \tilde{t}_{k+1})} a_k || M_{k+1}^{(2)}|| \nonumber \\
\leq& K  \big[\sup_{n_0 \leq k \leq n-1} \epsilon_k b_k || M_{k+1}^{(2)}|| \big] \sum_{k={n_0}}^{n-1} e^{-\kappa (\tilde{t}_n - \tilde{t}_{k+1})} a_k \nonumber \\
\leq& K \big[\sup_{n_0 \leq k \leq n-1} \epsilon_k b_k || M_{k+1}^{(2)}|| \big], \nonumber 
\end{align}
\begin{align}
&K\sum_{k={n_0}}^{n-1} \int_{\tilde{t}_k}^{\tilde{t}_{k+1}} \epsilon_k b_k ||\Phi(\tilde{t}_n, s, \overline{z}(s), y(s))|||| M_{k+1}^{(2)}||^2ds \nonumber \\
\leq&K \big[\sup_{n_0 \leq k \leq n-1} \epsilon_k b_k || M_{k+1}^{(2)}||^2 \big]. \nonumber
\end{align}
Combining all of the above bounds, we have
\begin{align}
||D_n|| \leq K \bigg[\sup_{n_0 \leq k \leq n-1} \epsilon_k b_k &+ \sup_{n_0 \leq k \leq n-1} \epsilon_k b_k || M_{k+1}^{(2)}|| \nonumber \\
&+ \sup_{n_0 \leq k \leq n-1} \epsilon_k b_k || M_{k+1}^{(2)}||^2 \bigg]. \nonumber
\end{align}
\end{proof}
Combining the above, we have
\begin{align}
||&\overline{z}(\tilde{t}_n) - x(\tilde{t}_n)|| \leq K \bigg[ e^{-\kappa_x (\tilde{t}_n-\tilde{t}_{n_0})}||\overline{z}(\tilde{t}_{n_0}) - x(\tilde{t}_{n_0})|| \nonumber \\ 
& +\sup_{{n_0} \leq k \leq n-1} b_k + \sup_{{n_0} \leq k \leq n-1} b_k ||M_{k+1}^{(2)}|| \nonumber \\
& + \sup_{{n_0} \leq k \leq n-1} b_k^2 + \sup_{{n_0} \leq k \leq n-1} b_k^2||M_{k+1}^{(2)}||  \nonumber \\
&+  \sup_{{n_0} \leq k \leq n-1} b_k^2||M_{k+1}^{(2)}||^2 + \sup_{{n_0} \leq k \leq n-1} \epsilon_k || M_{k+1}^{(2)} ||   \nonumber \\
& + \sup_{{n_0} \leq k \leq n-1} \epsilon_k +\sup_{n_0 \leq k \leq n-1} \epsilon_k b_k  \nonumber \\
&+ \sup_{n_0 \leq k \leq n-1} \epsilon_k b_k || M_{k+1}^{(2)}|| + \sup_{n_0 \leq k \leq n-1} \epsilon_k b_k || M_{k+1}^{(2)}||^2 \bigg]. \nonumber 
\end{align}
Combining with the results of the preceding subsections, we have the following error bound on $G_n$:
\begin{align}
||x_n - &z_n|| \leq K \bigg[ ||S_n^{(1)}|| + \sup_{{n_0} \leq k \leq n-1} a_k   \nonumber \\
& + \sup_{{n_0} \leq k \leq n-1} a_k ||M_{k+1}^{(1)}||^2 + \sup_{{n_0} \leq k \leq n-1} b_k  \nonumber \\
& + \sup_{{n_0} \leq k \leq n-1} b_k ||M_{k+1}^{(2)}|| + \sup_{{n_0} \leq k \leq n-1} b_k^2   \nonumber \\
&+ \sup_{{n_0} \leq k \leq n-1} b_k^2||M_{k+1}^{(2)}|| +  \sup_{{n_0} \leq k \leq n-1} b_k^2||M_{k+1}^{(2)}||^2 \nonumber \\
& + \sup_{{n_0} \leq k \leq n-1} \epsilon_k || M_{k+1}^{(2)} || + \sup_{{n_0} \leq k \leq n-1} \epsilon_k  \nonumber \\
&  +\sup_{n_0 \leq k \leq n-1} \epsilon_k b_k + \sup_{n_0 \leq k \leq n-1} \epsilon_k b_k || M_{k+1}^{(2)}|| \nonumber \\
&+ \sup_{n_0 \leq k \leq n-1} \epsilon_k b_k || M_{k+1}^{(2)}||^2  + e^{-\kappa_x (\tilde{t}_n-\tilde{t}_{n_0})}H_{n_0} \bigg], \nonumber 
\end{align}
where $H_{n_0} = \big(||\overline{x}(\tilde{t}_{n_0}) - x(\tilde{t}_{n_0})|| + ||\overline{z}(\tilde{t}_{n_0}) - x(\tilde{t}_{n_0})|| \big)$. Using $||x|| \leq 1 + ||x||^2$, we have:  on $G_n$,
\begin{align}
||x_n - &z_n|| \leq K \bigg[ ||S_n^{(1)}|| +e^{-\kappa_x (\tilde{t}_n-\tilde{t}_{n_0})} H_{n_0} \nonumber \\
&+ \sup_{{n_0} \leq k \leq n-1} a_k + \sup_{{n_0} \leq k \leq n-1} a_k ||M_{k+1}^{(1)}||^2  \nonumber \\
& + \sup_{{n_0} \leq k \leq n-1} b_k + \sup_{{n_0} \leq k \leq n-1} b_k ||M_{k+1}^{(2)}||^2 \nonumber \\
& + \sup_{{n_0} \leq k \leq n-1} b_k^2 + \sup_{{n_0} \leq k \leq n-1} b_k^2||M_{k+1}^{(2)}||^2  \nonumber \\
&+ \sup_{{n_0} \leq k \leq n-1} \epsilon_k + \sup_{{n_0} \leq k \leq n-1} \epsilon_k || M_{k+1}^{(2)} ||^2   \nonumber \\
& +\sup_{n_0 \leq k \leq n-1} \epsilon_k b_k + \sup_{n_0 \leq k \leq n-1} \epsilon_k b_k || M_{k+1}^{(2)}||^2  \bigg]. \nonumber \\
&\leq K \bigg[ ||S_n^{(1)}|| +e^{-\kappa_x (\tilde{t}_n-\tilde{t}_{n_0})} H_{n_0} \nonumber  \\
&+ \sup_{{n_0} \leq k \leq n-1} a_k   + \sup_{{n_0}\leq k \leq n-1} a_k ||M_{k+1}^{(1)}||^2 \nonumber  \\
&+ \sup_{{n_0} \leq k \leq n-1} \epsilon_k + \sup_{{n_0} \leq k \leq n-1} \epsilon_k || M_{k+1}^{(2)} ||^2 \bigg]. \label{numberone}
\end{align}

\subsection{Deviation bound for  $\{y_n\}$}

Now let $\hat{t}_0 = 0, \hat{t}_{k+1} = \hat{t}_k + b_k, k \geq 0$.
Rewrite (\ref{eq:stoc_approx_y}) as
\begin{align}
y_{k+1} = y_{k} &+ b_k \big( g(\lambda(y_k),y_k) \nonumber \\
&+  (g(x_k,y_k) - g(\lambda(y_k),y_k)) + M_{k+1}^{(2)} \big). \nonumber
\end{align}
Define  $\overline{y}(\cdot)$ by $\overline{y}(\hat{t}_k) = y_k \ \forall k$ and for $t \in (\hat{t}_{k}, \hat{t}_{k+1})$, 
\begin{align*}
\overline{y}(t) = y_k + \frac{t-\hat{t}_k}{b_k} [ y_{k+1} - y_k ].
\end{align*}
Then
\begin{align}
\overline{y}(t) = &\overline{y}(\hat{t}_{n_0}) + \int_{\hat{t}_{n_0}}^tg(\lambda(\bar{y}(s)), \bar{y}(s))ds \nonumber \\
& + \int_{\hat{t}_{n_0}}^{t} \big( \xi_7(s) + \xi_8(s) + \xi_9(s) \big) ds, \nonumber
\end{align}
where for $s \in [\hat{t}_k,\hat{t}_{k+1})$,
\begin{align}
\xi_7(s) &= g(\lambda(y_k),y_k) - g(\lambda(\overline{y}(s)),\overline{y}(s)),  \nonumber \\
\xi_8(s) &= g(x_k,y_k) - g(\lambda(y_k),y_k),  \ \xi_9(s) = M_{k+1}^{(2)}. \nonumber
\end{align}
This can be seen as a perturbation of the differential equation:
\begin{align}
\dot{y}(t) = g(\lambda(y(t)), y(t)). \nonumber
\end{align}
The generalized Alekseev's formula yields:
\begin{align}
\overline{y}(t) = y(t, &\hat{t}_{n_0}, y(\hat{t}_{n_0})) + \Phi_y(t, \hat{t}_{n_0}, \overline{y}(\hat{t}_{n_0}))(\overline{y}(\hat{t}_{n_0}) - y(\hat{t}_{n_0})) \nonumber \\
&+\int_{\hat{t}_{n_0}}^{t}\Phi_y(t, s, \overline{y}(s) ) \big[ \xi_1(s) + \xi_2(s) + \xi_3(s) \big]ds, \nonumber
\end{align}
where $y(t) \equiv y^*$ and
$\Phi_y(\cdot)$ is given by
\begin{align}
\label{eq:Phi_diff_eq_y}
\dot{\Phi}_y(t, s,  y_0) = \tilde{D}(\lambda(y^*), y^*) \Phi_y(t, s, y_0)
\end{align}
with $\Phi_y(s, s, y_0) = \mathcal{I}$, $\tilde{D}$ being the Jacobian matrix of $g(\lambda(\cdot), \cdot)$. 
As shown in Lemma $5.3$, \cite{gugan}, there exists $K, \kappa_y > 0$ so that the following holds for $t \geq s$:
\begin{align}
|| \Phi_y(t, s, y_0) || \leq K e^{-\kappa_y (t-s)}. \nonumber
\end{align} 
Define the following:
\begin{align}
\hat{A}_n &= \sum_{k={n_0}}^{n-1} \int_{\hat{t}_k}^{\hat{t}_{k+1}}\Phi_y(\hat{t}_n, s, \overline{y}(s)) \big[ g(\lambda(y_k),y_k) \nonumber \\
& \qquad \qquad \qquad \qquad \qquad - g(\lambda(\overline{y}(s)),\overline{y}(s)) \big] ds, \nonumber \\
\hat{B}_n &= \sum_{k={n_0}}^{n-1} \int_{\hat{t}_k}^{\hat{t}_{k+1}}\Phi_y(\hat{t}_n, s, \overline{y}(s)) \big[ g(x_k,y_k) \nonumber \\
& \qquad \qquad \qquad \qquad \qquad - g(\lambda(y_k),y_k) \big] ds, \nonumber \\
\hat{C}_n &= \sum_{k={n_0}}^{n-1} \int_{\hat{t}_k}^{\hat{t}_{k+1}}\Phi_y(\hat{t}_n, s, \overline{y}(s))M_{k+1}^{(2)} ds, \nonumber \\
\hat{S}_n^{(2)} &= \sum_{k={n_0}}^{n-1} \int_{\hat{t}_k}^{\hat{t}_{k+1}}\Phi_y(\hat{t}_n, s, \overline{y}(\hat{t}_k))M_{k+1}^{(2)} ds. \nonumber 
\end{align}
Then
\begin{align}
&||\overline{y}(\hat{t}_n) - y(\hat{t}_n, \hat{t}_{n_0}, y_{n_0})|| \nonumber \\
&\leq ||\Phi_y(t, \hat{t}_{n_0}, \overline{y}(\hat{t}_{n_0}))(\overline{y}(\hat{t}_{n_0}) - y(\hat{t}_{n_0}))|| + ||\hat{A}_n||  + ||\hat{B}_n|| \nonumber \\
&\qquad \qquad \qquad \qquad \qquad + ||\hat{C}_n-\hat{S}_n^{(2)}|| + ||\hat{S}_n^{(2)}||. \nonumber
\end{align}
As done in the previous two sections, we have:
\begin{align}
||\Phi_y(t, \hat{t}_{n_0}, \overline{y}(\hat{t}_{n_0}))&(\overline{y}(\hat{t}_{n_0}) - y(\hat{t}_{n_0}))|| \nonumber \\
&\leq e^{-\kappa_y (\hat{t}_n-\hat{t}_{n_0})}( ||\overline{y}(\hat{t}_{n_0}) - y(\hat{t}_{n_0})||). \nonumber 
\end{align}
We now bound each of the other terms on the right hand side through a sequence of lemmas.

\begin{lemma}
Let $n \geq n_0$ be arbitrary. Then on $G_n$ (now redefined in terms of $\{y_n\}$),
\begin{align}
||\hat{A}_n|| \leq K \big[ \sup_{{n_0} \leq k \leq n-1} b_k + \sup_{{n_0} \leq k \leq n-1} b_k ||M_{k+1}^{(2)}|| \big]. \nonumber
\end{align}

\end{lemma}
\begin{proof}
The proof exactly follows that of Lemma $5.6$, \cite{gugan}
\end{proof}

\begin{lemma}
Let $n \geq n_0$ be arbitrary. Then on $G_n$:
\begin{align}
||\hat{B}_n|| &\leq K \bigg[ \sup_{{n_0} \leq k \leq n-1}||S_k^{(1)}|| + \sup_{{n_0} \leq k \leq n-1} \epsilon_k H_{n_0} \nonumber \\
&+ \sup_{{n_0} \leq k \leq n-1} a_k + \sup_{{n_0} \leq k \leq m-1} a_k ||M_{k+1}^{(1)}||^2 \nonumber \\ 
&+ \sup_{{n_0} \leq k \leq n-1} \epsilon_k + \sup_{{n_0} \leq k \leq n-1} \epsilon_k || M_{k+1}^{(2)} ||^2  \bigg]. \nonumber  \nonumber
\end{align}
\end{lemma}

\begin{proof}
We have $||\hat{B}_n||$ bounded from above by 
\begin{align}
&\leq \sum_{k={n_0}}^{n-1} \int_{\hat{t}_k}^{\hat{t}_{k+1}} ||\Phi_y(\hat{t}_n, s, \overline{y}(s))|| \times \nonumber \\
&\qquad \qquad \qquad \qquad ||  g(x_k,y_k) - g(\lambda(y_k),y_k) ||ds \nonumber \\
&\leq L_g \sum_{k={n_0}}^{n-1} \int_{\hat{t}_k}^{\hat{t}_{k+1}} ||\Phi_y(\hat{t}_n, s, \overline{y}(s))|||| x_k - \lambda(y_k)|| ds \nonumber \\
&=  L_g \sum_{k={n_0}}^{n-1} \int_{\hat{t}_k}^{\hat{t}_{k+1}} ||\Phi_y(\hat{t}_n, s, \overline{y}(s))||||  x_k - z_k ||ds. \nonumber
\end{align}
From the previous section we have:
\begin{align}
||  x_k - z_k || &\leq K \bigg[ ||S_k^{(1)}|| +e^{-\kappa_x (\tilde{t}_k-\tilde{t}_{n_0})} H_{n_0} \nonumber \\
&+ \sup_{{n_0} \leq m \leq k-1} a_m + \sup_{{n_0} \leq m \leq k-1} a_m ||M_{m+1}^{(1)}||^2 \nonumber \\ 
&+ \sup_{{n_0} \leq m \leq k-1} \epsilon_m + \sup_{{n_0} \leq m \leq k-1} \epsilon_m || M_{m+1}^{(2)} ||^2 \bigg]. \nonumber
\end{align}
Substituting this back into the inequality, we have:
\begin{align}
&||\hat{B}_n|| \leq L_g K\sum_{k={n_0}}^{n-1} \int_{\hat{t}_k}^{\hat{t}_{k+1}} ||\Phi_y( \hat{t}_n, s, \overline{y}(s))||\bigg(||S_k^{(1)}|| \nonumber \\
&+ \sup_{{n_0} \leq m \leq k-1} a_m + \sup_{{n_0} \leq m \leq k-1} a_m ||M_{m+1}^{(1)}||^2\nonumber \\
&+ \sup_{{n_0} \leq m \leq k-1} \epsilon_m + \sup_{{n_0} \leq m \leq k-1} \epsilon_m || M_{m+1}^{(2)} ||^2  \nonumber \\
&+ e^{-\kappa_x (\tilde{t}_k-\tilde{t}_{n_0})} H_{n_0} \bigg) ds\nonumber \\
& \leq K \bigg[ \sup_{{n_0} \leq k \leq n-1}||S_k^{(1)}|| + \sup_{{n_0} \leq k \leq n-1} \epsilon_k H_{n_0} \nonumber \\
&+ \sup_{{n_0} \leq k \leq n-1} a_k + \sup_{{n_0} \leq k \leq m-1} a_k ||M_{k+1}^{(1)}||^2 \nonumber \\ 
&+ \sup_{{n_0} \leq k \leq n-1} \epsilon_k + \sup_{{n_0} \leq k \leq n-1} \epsilon_k || M_{k+1}^{(2)} ||^2  \bigg]. \nonumber 
\end{align}
The term $\sup_{{n_0} \leq k \leq n-1} \epsilon_k H_{n_0}$is derived as follows:
\begin{align}
&H_{n_0} \sum_{k={n_0}}^{n-1} \int_{\hat{t}_k}^{\hat{t}_{k+1}} ||\Phi_y( \hat{t}_n, s, \overline{y}(s))|| e^{-\kappa_x (\tilde{t}_k-\tilde{t}_{n_0})}ds \nonumber \\
&\leq  K H_{n_0}\sum_{k={n_0}}^{n-1} \int_{\hat{t}_k}^{\hat{t}_{k+1}} e^{-\kappa_x (\tilde{t}_k-\tilde{t}_{n_0})}ds \nonumber \\
&\leq K H_{n_0} \sup_{{n_0} \leq m \leq k-1} \epsilon_k. \nonumber
\end{align}
\end{proof}

\begin{lemma}
Let $n \geq n_0$ be arbitrary. Then on $G_n$,
\begin{align}
||\hat{C}_n - \hat{S}_n^{(2)}|| \leq K \big[ \sup_{{n_0} \leq k \leq n-1} &b_k ||M_{k+1}^{(2)}|| \nonumber \\
&+ \sup_{{n_0} \leq k \leq n-1} b_k ||M_{k+1}^{(2)}||^2 \big]. \nonumber
\end{align}
\end{lemma}
\begin{proof}
The proof exactly follows that of Lemma $5.7$, \cite{gugan}
\end{proof}
\hfill\break
Combining, we have the bound: on $G_n$,
\begin{align}
||\overline{y}(\hat{t}_n) &- y(\hat{t}_n)|| \leq K  \bigg[ ||\hat{S}_n^{(2)}|| + \sup_{{n_0} \leq k \leq n-1}||S_k^{(1)}||  + \nonumber \\
& \sup_{{n_0} \leq k \leq n-1} b_k + \sup_{{n_0} \leq k \leq n-1} b_k ||M_{k+1}^{(2)}|| + \nonumber \\
 &\sup_{{n_0} \leq k \leq n-1} a_k + \sup_{{n_0} \leq k \leq m-1} a_k ||M_{k+1}^{(1)}||^2 + \nonumber \\ 
&\sup_{{n_0} \leq k \leq n-1} \epsilon_k + \sup_{{n_0} \leq k \leq n-1} \epsilon_k || M_{k+1}^{(2)} ||^2 + \nonumber \\
&\sup_{{n_0} \leq k \leq n-1} b_k ||M_{k+1}^{(2)}|| + \sup_{{n_0} \leq k \leq n-1} b_k ||M_{k+1}^{(2)}||^2 + \nonumber \\
& e^{-\kappa_y (\hat{t}_n-\hat{t}_{n_0})} ||\overline{y}(\hat{t}_{n_0}) - y(\hat{t}_{n_0})|| + \sup_{{n_0} \leq k \leq n-1} \epsilon_k H_{n_0} \bigg]. \nonumber
\end{align}
Using $||x|| \leq 1 + ||x||^2$ and that $b_k \leq \epsilon_k \ \forall k$ gives us:
\begin{align}
||\overline{y}&(\hat{t}_n) - y(\hat{t}_n)|| \leq K \bigg[ ||\hat{S}_n^{(2)}|| + \sup_{{n_0} \leq k \leq n-1}||S_k^{(1)}||  \nonumber \\
&+\sup_{{n_0} \leq k \leq n-1} a_k + \sup_{{n_0} \leq k \leq m-1} a_k ||M_{k+1}^{(1)}||^2  \nonumber \\ 
&+ \sup_{{n_0} \leq k \leq n-1} \epsilon_k + \sup_{{n_0} \leq k \leq n-1} \epsilon_k || M_{k+1}^{(2)} ||^2  \nonumber \\
&+ e^{-\kappa_y (\hat{t}_n-\hat{t}_{n_0})} ||\overline{y}(\hat{t}_{n_0}) - y(\hat{t}_{n_0})|| + \sup_{{n_0} \leq k \leq n-1} \epsilon_k H_{n_0}\bigg]. \label{numbertwo}
\end{align}

\section{Concentration bounds}

\subsection{Concentration of $|| x_n - z_n ||$}

For $\epsilon > 0$. Let $N$ be such that for all $n \geq N$, we have $a_n \leq \frac{\epsilon}{8K}, \epsilon_n \leq \frac{\epsilon}{8K}$ $\forall n \geq N$. For $n_0 \geq N$ and $K$ as in (\ref{numberone}), let $T$ be such that:
\begin{align}
e^{-\kappa_x (\tilde{t}_n-\tilde{t}_{n_0})} H_{n_0} \leq \frac{\epsilon}{8K} \: \: \forall \: n \geq n_0 + T. \nonumber
\end{align}
From (\ref{numberone}) and Lemma $3.1$ \cite{gugan}, we have
\begin{align}
\mathbb{P}( ||x_n - z_n|| \leq \epsilon \: &\forall n \geq n_0 + T + 1 | x_{n_0}, z_{n_0} \in B )  \nonumber \\
&\geq  1 - \mathbb{P}_0\bigg(  \bigcup_{n = n_0}^{\infty}\{G_n, ||S_n^{(1)}|| > \frac{\epsilon}{8K} \} \nonumber \\
&\quad \cup \:  \bigcup_{n = n_0}^{\infty}\{G_n, a_n ||M_{n+1}^{(1)}||^2 > \frac{\epsilon}{8K} \} \nonumber \\
&\quad \cup \: \bigcup_{n = n_0}^{\infty}\{G_n, \epsilon_n ||M_{n+1}^{(2)}||^2 > \frac{\epsilon}{8K} \}  \bigg), \nonumber
\end{align}
where $\mathbb{P}_0(\cdot)$ denotes the conditional probability given $ x_{n_0}, z_{n_0} \in B$. Using the union bound, we have
\begin{align}
\mathbb{P}( ||x_n - &z_n|| \leq \epsilon \: \forall n \geq n_0 + T + 1 | x_{n_0}, z_{n_0} \in B )  \nonumber \\
&\geq  1 - \sum_{n=n_0}^{\infty}\mathbb{P}_0(G_n, ||S_n^{(1)}|| > \frac{\epsilon}{8K})  \nonumber \\
& \quad - \sum_{n=n_0}^{\infty}\mathbb{P}_0(G_n, a_n ||M_{n+1}^{(1)}||^2 > \frac{\epsilon}{8K}) \nonumber \\
& \quad - \sum_{n=n_0}^{\infty}\mathbb{P}_0(G_n, \epsilon_n ||M_{n+1}^{(2)}||^2 > \frac{\epsilon}{8K}). \nonumber
\end{align}
From Theorem $6.2$, \cite{gugan}, we have: for some $K_1 > 0,$
\begin{align}
\sum_{n=n_0}^{\infty}\mathbb{P}_0(G_n, a_n ||M_{n+1}^{(1)}||^2 > \frac{\epsilon}{8K}) &\leq K_1 \sum_{n=n_0}^{\infty}\exp \bigg(- \frac{K^2 \sqrt{\epsilon}}{\sqrt{a_n}} \bigg), \nonumber \\
\sum_{n=n_0}^{\infty}\mathbb{P}_0(G_n, \epsilon_n ||M_{n+1}^{(2)}||^2 > \frac{\epsilon}{8K}) &\leq K_1 \sum_{n=n_0}^{\infty}\exp \bigg(- \frac{K^2 \sqrt{\epsilon}}{\sqrt{\epsilon_n}} \bigg). \nonumber
\end{align}
From Theorem $6.3$, \cite{gugan}, for some $K_2, K_3 > 0$,  $\epsilon > 0$,
\begin{align}
\sum_{n=n_0}^{\infty}&\mathbb{P}_0(G_n, ||S_n^{(1)}|| > \frac{\epsilon}{8K}) \leq K^3 \sum_{n = n_0}^{\infty} \exp \bigg(- \frac{K_3 \epsilon^2}{\beta_n} \bigg). \nonumber 
\end{align}
where  $\beta_n := \max_{n_0 \leq k \leq n-1} \bigg[ e^{-\kappa_x \sum_{i=k+1}^{n-1} a_i}a_k \bigg]$. 
Combining these results, we have the following theorem- \\
\begin{theorem}
For $n_0, T$ defined above, we have the following concentration bound for a suitable $C_1 > 0$. 
For $\epsilon \leq 1$,
\begin{align}
&\mathbb{P}( ||x_n - z_n|| \leq \epsilon \: \forall n \geq n_0 + T + 1 | x_{n_0}, z_{n_0} \in B )  \nonumber \\
&\geq 1 - \sum_{n = n_0}^{\infty} C_1 \exp \bigg(- \frac{C_2 \sqrt{\epsilon}}{\sqrt{a_n}} \bigg)  \nonumber \\
& \quad - \sum_{n = n_0}^{\infty} C_1 \exp \bigg(- \frac{C_2 \sqrt{\epsilon}}{\sqrt{\epsilon_n}} \bigg)- \sum_{n = n_0}^{\infty} C_1 \exp \bigg(- \frac{C_2 \epsilon^2}{\beta_n} \bigg). \nonumber
\end{align}
For $\epsilon > 1$,
\begin{align}
&\mathbb{P}( ||x_n - z_n|| \leq \epsilon \: \forall n \geq n_0 + T + 1 | x_{n_0}, z_{n_0} \in B )  \nonumber \\
&\geq 1 - \sum_{n = n_0}^{\infty} C_1 \exp \bigg(- \frac{C_2 \sqrt{\epsilon}}{\sqrt{a_n}} \bigg)  \nonumber \\
& \quad - \sum_{n = n_0}^{\infty} C_1 \exp \bigg(- \frac{C_2 \sqrt{\epsilon}}{\sqrt{\epsilon_n}} \bigg) - \sum_{n = n_0}^{\infty} C_1 \exp \bigg(- \frac{C_2 \epsilon}{\beta_n} \bigg). \nonumber
\end{align}
\end{theorem}

\subsection{Concentration of $|| y_n - y(t_n)|| $}

Let $N, T$ be as before for $K$ as in (\ref{numbertwo}), and with
\begin{align}
e^{-\kappa_y (\hat{t}_n-\hat{t}_{n_0})}( ||\overline{y}(\hat{t}_{n_0}) &- y(\hat{t}_{n_0})||) \leq \frac{\epsilon}{8K}, \: \: \forall \: n \geq n_0 + T. \nonumber
\end{align}
Using (\ref{numbertwo}) and Lemma $3.1$ \cite{gugan} we have
\begin{align}
\mathbb{P}( ||y_n - &y(\hat{t}_n)|| \leq \ \epsilon \: \forall n \geq n_0 + T + 1 |x_{n_0}, y_{n_0}, z_{n_0} \in B) \nonumber \\
&\geq  1 - \mathbb{P}_1\bigg(  \bigcup_{n = n_0}^{\infty}\{G_n, ||S_n^{(1)}|| > \frac{\epsilon}{8K} \}  \nonumber \\
& \quad \cup \: \bigcup_{n = n_0}^{\infty}\{G_n, ||\hat{S}_n^{(2)}|| > \frac{\epsilon}{8K} \} \nonumber \\
& \quad \cup \:  \bigcup_{n = n_0}^{\infty}\{G_n, a_n ||M_{n+1}^{(1)}||^2 > \frac{\epsilon}{8K} \} \nonumber \\
&  \quad \cup \: \bigcup_{n = n_0}^{\infty}\{G_n, \epsilon_n ||M_{n+1}^{(2)}||^2 > \frac{\epsilon}{8K} \}  \bigg) \nonumber
\end{align}
where we use $\mathbb{P}_1(\cdot)$ to denote the conditional probability given $x_{n_0}, y_{n_0}, z_{n_0} \in B$. Using the union bound, 
\begin{align}
\mathbb{P}( ||y_n - &y(\hat{t}_n)|| \leq \ \epsilon \: \forall n \geq n_0 + T + 1 | x_{n_0}, y_{n_0}, z_{n_0} \in B) \nonumber \\
&\geq  1 - \sum_{n=n_0}^{\infty}\mathbb{P}_1(G_n, ||S_n^{(1)}|| > \frac{\epsilon}{8K}) \nonumber \\
& \quad -\sum_{n=n_0}^{\infty}\mathbb{P}_1(G_n, ||\hat{S}_n^{(2)}|| > \frac{\epsilon}{8K})  \nonumber \\
& \quad- \sum_{n=n_0}^{\infty}\mathbb{P}_1(G_n, a_n ||M_{n+1}^{(1)}||^2 > \frac{\epsilon}{8K}) \nonumber \\
& \quad - \sum_{n=n_0}^{\infty}\mathbb{P}_1(G_n, \epsilon_n ||M_{n+1}^{(2)}||^2 > \frac{\epsilon}{8K}). \nonumber
\end{align}
From Theorem $6.2$, \cite{gugan}, we have
\begin{align}
\sum_{n=n_0}^{\infty}\mathbb{P}_1(G_n, a_n ||M_{n+1}^{(1)}||^2 > \frac{\epsilon}{8K}) &\leq K_1' \sum_{n=n_0}^{\infty}\exp \bigg(- \frac{K_2' \sqrt{\epsilon}}{\sqrt{a_n}} \bigg), \nonumber \\
\sum_{n=n_0}^{\infty}\mathbb{P}_1(G_n, \epsilon_n ||M_{n+1}^{(2)}||^2 > \frac{\epsilon}{8K}) &\leq K_1' \sum_{n=n_0}^{\infty}\exp \bigg(- \frac{K_2' \sqrt{\epsilon}}{\sqrt{\epsilon_n}} \bigg). \nonumber
\end{align}
\\ \\
From Theorem $6.3$, \cite{gugan},  for suitable constants $\{K_i'\}$ and $\gamma_n := \max_{n_0 \leq k \leq n-1} \bigg[ e^{-\kappa_y \sum_{i=k+1}^{n-1} b_i}b_k \bigg]$ : for $\epsilon \leq 1$,
\begin{align}
\sum_{n=n_0}^{\infty}\mathbb{P}_1(G_n, ||S_n^{(1)}|| > \frac{\epsilon}{8K}) &\leq K_3' \sum_{n = n_0}^{\infty} \exp \bigg(- \frac{K_4' \epsilon^2}{\beta_n} \bigg), \nonumber \\
\sum_{n=n_0}^{\infty}\mathbb{P}_1(G_n, ||\hat{S}_n^{(2)}|| > \frac{\epsilon}{8K}) &\leq K_3' \sum_{n = n_0}^{\infty} \exp \bigg(- \frac{K_4' \epsilon^2}{\gamma_n} \bigg), \nonumber
\end{align}
and for $\epsilon > 1$,
\begin{align}
\sum_{n=n_0}^{\infty}\mathbb{P}_1(G_n, ||S_n^{(1)}|| > \frac{\epsilon}{8K}) &\leq K_3' \sum_{n = n_0}^{\infty} \exp \bigg(- \frac{K_4' \epsilon}{\beta_n} \bigg), \nonumber \\
\sum_{n=n_0}^{\infty}\mathbb{P}_1(G_n, ||\hat{S}_n^{(2)}|| > \frac{\epsilon}{8K}) &\leq K_3' \sum_{n = n_0}^{\infty} \exp \bigg(- \frac{K_4' \epsilon}{\gamma_n} \bigg). \nonumber
\end{align}
Combining these results, we have the following theorem: \\
\begin{theorem}
For $n_0, T$ as above and suitable constants $C_1, C_2 > 0$, we have: 
For $\epsilon \leq 1$,
\begin{align}
&\mathbb{P}( ||y_n - y(\hat{t}_n)|| \leq \ \epsilon \: \forall n \geq n_0 + T + 1 | x_{n_0}, y_{n_0}, z_{n_0} \in B) \geq \nonumber \\
& 1 - \sum_{n = n_0}^{\infty} C_1 \exp \bigg(- \frac{C_2 \sqrt{\epsilon}}{\sqrt{a_n}} \bigg)  - \sum_{n = n_0}^{\infty} C_1 \exp \bigg(- \frac{C_2 \sqrt{\epsilon}}{\sqrt{\epsilon_n}} \bigg) \nonumber \\
& - \sum_{n = n_0}^{\infty} C_1 \exp \bigg(- \frac{C_2 \epsilon^2}{\beta_n} \bigg) - \sum_{n = n_0}^{\infty} C_1 \exp \bigg(- \frac{C_2 \epsilon^2}{\gamma_n} \bigg). \nonumber 
\end{align}
For $\epsilon > 1$,
\begin{align}
&\mathbb{P}( ||y_n - y(\hat{t}_n)|| \leq \ \epsilon \: \forall n \geq n_0 + T + 1 | x_{n_0}, y_{n_0}, z_{n_0} \in B) \geq \nonumber \\
&1 - \sum_{n = n_0}^{\infty} C_1 \exp \bigg(- \frac{C_2 \sqrt{\epsilon}}{\sqrt{a_n}} \bigg)  - \sum_{n = n_0}^{\infty} C_1 \exp \bigg(- \frac{C_2 \sqrt{\epsilon}}{\sqrt{\epsilon_n}} \bigg) \nonumber \\
& - \sum_{n = n_0}^{\infty} C_1 \exp \bigg(- \frac{C_2 \epsilon}{\beta_n} \bigg) - \sum_{n = n_0}^{\infty} C_1 \exp \bigg(- \frac{C_2 \epsilon}{\gamma_n} \bigg). \nonumber 
\end{align}

\end{theorem}

\section{Conclusion}

We have derived a concentration bound for two time scale stochastic approximation that is valid for all time after a given time instant, extending the results of \cite{gugan} for the classical case. This was achieved by leveraging the stability properties of the limiting ODE by means of Alekseev's variation of constants formula. Future directions include extending this to the so called `Markov noise'. It also appears possible to exploit additional structure of specific stochastic approximation schemes to improve upon the bounds that have been derived here in a very general framework.


\end{document}